\documentclass[12pt,a4paper, oneside]{amsart}
\usepackage{graphicx}
\usepackage{multirow}
\usepackage{amsmath,amssymb,amsfonts}
\usepackage{amsthm}
\usepackage{mathrsfs}
\usepackage{float}
\usepackage[title]{appendix}
\usepackage{xcolor}
\usepackage{textcomp}
\usepackage{manyfoot}
\usepackage{booktabs}
\usepackage{algorithm}
\usepackage{algorithmicx}
\usepackage{algpseudocode}
\usepackage{listings}
\usepackage{hyperref}
\usepackage{cleveref}
\usepackage{geometry}
\usepackage[pagewise]{lineno}

\newtheorem{theorem}{Theorem}[section]
\newtheorem{proposition}[theorem]{Proposition}
\newtheorem{corollary}[theorem]{Corollary}
\newtheorem{lemma}[theorem]{Lemma}
\newtheorem{example}[theorem]{Example}
\newtheorem{remark}[theorem]{Remark}
\newtheorem{definition}[theorem]{Definition}

\begin{document}
\title[Characterizations of some rotundity properties...]{Characterizations of some rotundity\\ properties in terms of farthest points}
\author{Arunachala Prasath C}
\address{Department of Mathematics, National Institute of Technology Tiruchirappalli,\\	Tiruchirappalli - 620015, Tamil Nadu, India}
\email{caparunachalam@gmail.com}
\author{Vamsinadh Thota}
\address{Department of Mathematics, National Institute of Technology Tiruchirappalli,\\	Tiruchirappalli - 620015, Tamil Nadu, India}
\email[corresponding author]{vamsinadh@nitt.edu}
\date{\today}
\begin{abstract}
We characterize rotund, uniformly rotund, locally uniformly rotund and compactly locally uniformly rotund spaces in terms of sets of (almost) farthest points from the unit sphere using the generalized diameter. For this we introduce few remotality properties using the sets of almost farthest points. As a consequence, we obtain some characterizations of the aforementioned rotundity properties in terms of existing proximinality notions.
\end{abstract}
\maketitle
\section{Introduction}\label{sec1}
Let $X$ be a real Banach space. The closed unit ball and the unit sphere of $X$ are denoted by $B_X$ and $S_X$ respectively. For a non-empty bounded subset $F$ of $X$, $x\in X$ and $\delta\geq 0$, we denote $r(F,x)=sup\{\|x-y\|:y\in F\}$ and $Q_F(x,\delta)=\left\{y\in F:\|x-y\|\geq r(F,x)-\delta\right\}$. We write $Q_{F}(x,0)$, the set of all farthest points from $F$ to $x$, as $Q_{F}(x)$. It is clear that for $\delta>0$, the set of almost farthest points from $F$ to $x$, $Q_F(x,\delta)$, is non-empty. If $Q_{F}(x)$ is non-empty then we say $F$ is remotal at $x$. The collection of all non-empty closed(respectively, closed and bounded) subsets of $X$ is denoted by $CL(X)$(respectively, $CB(X)$).

\par Various rotundity notions given in the following definition play a vital role in many branches of functional analysis, particularly in best approximation theory and geometry of Banach spaces.
\begin{definition}\label{rdef}
	The space $X$ is said to be
	\begin{enumerate}
		\item rotund \cite{Clarkson} if $\left\|\frac{x_1+x_2}{2}\right\|<1$ whenever $x_1,x_2\in S_X$ and $x_1\neq x_2$;
		\item uniformly rotund (in short, $UR$) \cite{Clarkson} if $\|x_n-y_n\|\to 0$ whenever $(x_{n})$ and $(y_n)$ are two sequences in $S_X$ such that $\left\|\frac{x_n+y_n}{2}\right\|\to1;$
		\item locally uniformly rotund\ (in short, LUR) \cite{Lovaglia} if $x_n\to x$ whenever $x\in S_X$ and $(x_n)$ is a sequence in $S_X$ such that $\left\|\frac{x_n+x}{2}\right\|\to1;$
		\item compactly locally uniformly rotund (in short, CLUR) \cite{Vlasov} if $(x_n)$ has a convergent subsequence  whenever $x\in S_X$ and $(x_n)$ is a sequence in $S_X$ such that $\left\|\frac{x_n+x}{2}\right\|\to1.$	
	\end{enumerate}
\end{definition}

Several characterizations of the above rotundity notions are available in the literature (see for instance \cite{SP,UC,thota}). In this paper efforts were made to obtain some characterizations for the rotundity properties, given in the  \Cref{rdef}, using the sets containing farthest points and(or) almost farthest points. For this we need the following notions.

\begin{definition}\label{defn}
	Let $A,F$ be non-empty subsets of $X$ and $F$ be bounded. Then the set $F$ is said to be
	\begin{enumerate}
		\item remotal \cite[page 486]{Pai} on $A$ if $Q_{F}(x)\neq\phi$ for each $x\in A;$
		\item uniquely remotal on $A$ if  $Q_{F}(x)$ is singleton for each $x\in A;$
		\item strongly remotal on $A$  if for every $x\in A$ and $\epsilon>0$ there exists a $\delta=\delta(\epsilon,x)>0$ such that, $Q_{F}(x,\delta)\subseteq Q_{F}(x)+\epsilon B_X;$
		\item strongly uniquely remotal (in short, SUR) on $A$ if $F$ is both uniquely remotal and strongly remotal on $A;$
		\item uniformly strongly uniquely remotal (in short, USUR) on $A$ if $F$ is uniquely remotal on $A$ and for every $\epsilon>0$ there exists a $\delta=\delta(\epsilon)>0$ such that, $Q_{F}(x,\delta)\subseteq Q_{F}(x)+\epsilon B_X$ for every $x\in A;$

		\item sup-compact \cite[page 486]{Pai} on $A$ if for each $x\in A$ and every maximizing sequence $(y_n)$ in $F$, i.e. a sequence satisfying $\|x-y_n\|\to r(F,x)$, has a subsequence that converges to an element in $F$.
	\end{enumerate}
\end{definition}

These remotality properties are analogous to the notions from the theory of best approximation, respectively proximinal, Chebyshev, strongly proximinal, strongly Chebyshev, uniformly strongly Chebyshev and approximatively compact. For the definitions and some related results on these notions we refer to \cite{UC, thota}.

The paper is organized as follows. In section 2, we present some preliminary results that are needed to prove our main results. First we observe basic implications among the remotality properties given in  \Cref{defn} and we also present some counter examples to demonstrate these implications are strict. Further, we relate some remotality properties with the convergence of sets of almost farthest points, which also exhibit the interrelations among these properties. In addition, we prove that the closed unit ball and the unit sphere possess the remotality properties in a similar way.

\par In section 3, we present some characterizations of rotund, uniformly rotund, locally uniformly rotund and compactly locally uniformly rotund spaces. In particular, we prove that the space is rotund if and only if the corresponding unit sphere is uniquely remotal at every non-zero element. We observe that the strong unique remotality (respectively, sup-compactness) of the unit sphere at every non-zero point is necessary and sufficient for the local uniform rotundity (respectively, compactly local uniform rotundity) of the space. Further, we also present an uniform version of this result which provides a characterization for uniformly rotund spaces, i.e. the space is uniformly rotund if and only if the unit sphere is uniformly strongly uniquely remotal on outside the closed unit ball. As a consequence of these results we obtain characterizations of uniform rotundity and locally uniform rotundity in terms of nearly best approximation sets as well.

\section{Preliminaries}\label{sec2}

We begin this section with some relations among the remotality properties.
\begin{remark}\label{implications}
The following implications are easy to verify.
\begin{enumerate}
\item USUR $\Rightarrow$ SUR $\Rightarrow$ Sup-compact $\Rightarrow$ Strongly remotal $\Rightarrow$ Remotal.

\item USUR $\Rightarrow$ SUR $\Rightarrow$ Uniquely remotal $\Rightarrow$ Remotal.
\end{enumerate}
However, none of these implications can be reversed in general.
\end{remark}
The following example illustrates that in general remotality does not imply any of the other properties given in \Cref{defn}.
\begin{example}
Let $X=(\ell_\infty, \|\cdot\|_\infty)$ and $x=(0,\frac{1}{2}, \frac{2}{3},\frac{3}{4},\ldots) \in S_X.$ Clearly, $B_X$ is remotal at $x.$ Let $z=(z_k) \in B_X,$ observe that $z \in Q_{B_X}(x)$ iff $(z_k)$ has a subsequence that converges to $-1.$ Thus, $B_X$ is not uniquely remotal. Note that $-e_{n} \in Q_{B_X}\left(x, \frac{1}{n}\right)$ for every $n \in \mathbb{N}.$ Hence, $(-e_n)$ is a maximizing sequence in $B_X$ for $x.$ Therefore, $B_X$ is not sup-compact at $x.$ In fact, since $d(-e_n, Q_{B_X}(x)) \geq 1$ for every $n \in \mathbb{N}$ it follows that $B_X$ is not strongly remotal at $x.$
\end{example}

The following example demonstrates that strong remotality is strictly weaker to the sup-compactness (hence, to SUR) in general.
\begin{example}
Let $X=(\ell_1, \|\cdot\|_1).$ For $-e_1 \in S_X,$ since $e_n \in Q_{S_X}(-e_1)$ for every $n \in \mathbb{N},$ it follows that $S_X$ is neither sup-compact nor uniquely remotal at $-e_1.$ Observe that for $ y=(y_k) \in S_X$ we have  $y \in Q_{S_X}(-e_1)$ if and only if $0\leq y_1\leq1.$ Thus, for any $w=(w_k) \in S_X$, we have $w'=(1-\sum\limits_{k=2}^\infty|w_k|, w_2, w_3,\ldots) \in  Q_{S_X}(-e_1).$ Further, for any $\delta > 0$ and $z=(z_k) \in Q_{S_X}(-e_1, \delta),$ notice that either $z\in Q_{S_X}(-e_1) $ or $|z_1| \leq \frac{\delta}{2}.$ Now, for any  $z=(z_k) \in Q_{S_X}(-e_1, \delta)$ define  $\overline{z}=\begin{cases}z', & \mbox{if } z_1 <0;\\ z, & \mbox{if }  z_1 \geq 0. \end{cases}$  It is easy to check that $\|z-\overline{z}\| \leq \delta$ for every $z\in Q_{S_X}(-e_1, \delta).$ Therefore, for $0 < \delta < \epsilon,$ we have $Q_{S_X}(-e_1, \delta) \subseteq Q_{S_X}(-e_1)+\epsilon B_X.$ i.e $S_X$ is strongly remotal at $x=-e_1.$
\end{example}

We remark that our main results will provide several examples to see the other implications in \Cref{implications} cannot be reversed in general. However, there exist some interrelations among these notions. To see this we need the following.




The Hausdorff distance between any pair of sets $A, B \in CB(X),$ denoted by $H(A,B)$, is defined as $H(A,B)=\inf\{r>0: A \subseteq B+rB_X \mbox{and } B \subseteq A+rB_X\}.$ The measure of non-compactness \cite{Banas} of a non-empty subset $A$ of $X$, is defined as $\alpha(A)=\inf\{\epsilon>0: A \subseteq E+\epsilon B_X \mbox{ for some finite subset } E \mbox{ of } X \}$.

\begin{definition}
Let $(C_n)$ be a sequence in $CB(X)$ and $C_0\in CB(X)$. We say the sequence $(C_n)$ converges to $C_0$ in the
	\begin{enumerate}
		\item upper Vietoris sense, denoted by $C_n\xrightarrow{V^+}\;C_0,$ if $C_n\subseteq U$ eventually whenever  $U$ is an open set in $X$ such that $C_0\subseteq U;$
		\item lower Vietoris sense, denoted by $C_n\xrightarrow{V^-}\;C_0,$ if $C_n\cap U\neq \phi$ eventually whenever  $U$ is an open set in $X$ such that $C_0\cap U\neq \phi;$
		\item Vietoris sense, denoted by $C_n\xrightarrow{V}\;C_0,$ if $C_n\xrightarrow{V^+}\;C_0$ and $C_n\xrightarrow{V^-}\;C_0;$
		\item upper Hausdorff sense, denoted by $C_n\xrightarrow{H^+}\;C_0,$ if for every $\epsilon >0$, $C_n\subseteq C_0+\epsilon B_X$ eventually;
		\item lower Hausdorff sense, denoted by $C_n\xrightarrow{H^-}\;C_0,$ if for every $\epsilon >0$, $C_0\subseteq C_n+\epsilon B_X$ eventually;	
		\item  Hausdorff sense, denoted by $C_n\xrightarrow{H}\;C_0,$ if $C_n\xrightarrow{H^+}\;C_0$ and $C_n\xrightarrow{H^-}\;C_0$.		
	\end{enumerate}
\end{definition}

\begin{theorem}\label{supcompact}
Let $F \in CB(X)$ be remotal at $x \in X$. Then the following statements are equivalent.
	\begin{enumerate}
		\item $F$ is sup-compact at $x$.
		\item $Q_F(x)$ is compact and $Q_F\left(x,\frac{1}{n}\right)  \xrightarrow{V} Q_F(x).$
		\item $Q_F(x)$ is compact and $Q_F\left(x,\frac{1}{n}\right) \xrightarrow{H} Q_F(x).$
		\item $Q_F(x)$ is compact and $F$ is strongly remotal at $x$.
		\item $\alpha\left(Q_F\left(x,\frac{1}{n}\right)\right)\to 0$.
	\end{enumerate}
\end{theorem}

\begin{proof}
$(1)\Rightarrow(2)$: Let $(y_n)$ be a sequence in $Q_F(x)$. Then, by assumption, $(y_n)$ has a convergent subsequence. Therefore, $Q_F(x)$ is compact. Since $\left(Q_F\left(x,\frac{1}{n}\right)\right)_{n \in \mathbb{N}}$ is a decreasing sequence of sets and $Q_F(x) = \cap_{n=1}^\infty Q_F\left(x,\frac{1}{n}\right)$, we have $Q_F\left(x,\frac{1}{n}\right) \xrightarrow{ V^-}   Q_F(x)$.

Suppose $Q_F\left(x,\frac{1}{n}\right) \not\xrightarrow{V^+} Q_F(x)$. Then there exist an open set $U$ of $X$ and a sequence $(y_n) \in F$ such that $Q_F(x)\subseteq U$, $y_n\in Q_F\left(x,\frac{1}{n}\right)$ but $y_n\notin U$ for every $n \in \mathbb{N}$. Since $(y_n)$ is a maximizing sequence in $F$ for $x$, by $(1)$, there exists a subsequence $(y_{n_k})$ of $(y_n)$ such that $(y_{n_k})$ converges to $y_0$ for some $y_0\in Q_F(x),$ which is a contradiction. Therefore, $Q_F\left(x,\frac{1}{n}\right)\xrightarrow{V^+}\;Q_F(x)$.

\noindent$(2)\Rightarrow (3)$: This implication follows from \cite[Theorems 10 and 14]{Chapter}.

\noindent$(3)\Rightarrow (4)$:  Let $\epsilon>0$. Then there exists $n_0\in \mathbb{N}$ such that  $ Q_F\left(x,\frac{1}{n}\right)\subseteq Q_F(x)+\epsilon B_X$ for every $n\geq n_0$. Therefore, for $\delta < \frac{1}{n_0}$, we have $Q_F\left(x,\delta \right)\subseteq Q_F(x)+\epsilon B_X.$

\noindent$(4) \Rightarrow (1)$: Let $(y_n)$ be a maximizing sequence in $F$ for $x$ and $\epsilon >0$. For every $n \in \mathbb{N}$, w.l.o.g assume, $y_n\in Q_F\left(x,\frac{1}{n}\right)$. By assumption, there exists $\delta>0$ such that $ Q_F\left(x,\delta \right)\subseteq Q_F(x)+\epsilon B_X.$ Thus, for  every $n\geq\frac{1}{\delta}$, we have $y_n \in Q_F(x)+\epsilon B_X.$  Therefore, for every  $n \in \mathbb{N}$ there exists  $x_n \in Q_F(x)$ such that $\|y_n-x_n\|<\epsilon.$ Since $Q_F(x)$ is compact,  $(x_n)$ has a convergent subsequence, which further implies $(y_n)$ has a convergent subsequence.

\noindent$(2)\Leftrightarrow (5)$: These implications follow from \cite[Theorem 15]{Chapter}.
\end{proof}

\begin{theorem}\label{remseq}
Let $F \in CB(X)$ be remotal at $x \in X$. Then the following statements are equivalent.
	\begin{enumerate}
		\item Every maximizing sequence in $F$ for $x$ converges.
		\item $diam\left(Q_F\left(x,\frac{1}{n}\right)\right)\to 0$.
		\item $Q_F(x)$ is singleton and $Q_F\left(x,\frac{1}{n}\right)  \xrightarrow{V} Q_F(x).$
		\item $Q_F(x)$ is singleton and $Q_F\left(x,\frac{1}{n}\right)  \xrightarrow{H} Q_F(x).$
		\item $F$ is SUR at $x$.
	\end{enumerate}
\end{theorem}

\begin{proof}
$(1)\Rightarrow (2)$: Clearly $Q_F(x)$ is singleton. Suppose $diam\left(Q_F\left(x,\frac{1}{n}\right)\right) \not\to 0$. Then, there exists $\epsilon>0$ such that $diam\left(Q_F\left(x,\frac{1}{n}\right)\right)>\epsilon$ for every $n\in \mathbb{N}$. Therefore there exist two sequences $(y_n), (y_n') \in F$ such that $y_n, y_n' \in Q_F\left(x,\frac{1}{n}\right)$ and $\|y_n-y_n'\|>\epsilon$ for every $n\in \mathbb{N}$. By $(1)$, both $(y_n)$ and $(y_n')$ converge to $y_0$ and $y_0' $ respectively for some $y_0, y_0' \in F.$ Observe that $y_0, y_0' \in Q_F(x).$ Since,  $Q_F(x)$ is singleton, we have $\|y_n-y_n'\|\to 0$ which is a contradiction. Thus, $diam\left(Q_F\left(x,\frac{1}{n}\right)\right)\to 0.$

\noindent$(2)\Rightarrow (3)\Rightarrow (4)$: These implications follow from \cite[Theorems 10 and 15]{Chapter}.

\noindent$(4)\Rightarrow (5)\Rightarrow(1)$: These implications follow in similar lines to the proof of $(3)\Rightarrow (4)\Rightarrow(1)$ in \Cref{supcompact}.
\end{proof}

The properties given in the following result are easy to verify.

\begin{proposition}\label{Q_F}
Let $F\in CB(X)$ and $x\in X$.
	\begin{enumerate}
		\item $r(F,x)=0\Leftrightarrow F=\{x\}$.
		\item $r(z+F,z+x)=r(F,x)$ for any $z\in X$.
		\item $r(\alpha F,\alpha x)=\vert \alpha \vert r(F,x)$ for any $\alpha \in \mathbb{R}$.
		\item $Q_F(x,\delta_1)\subseteq Q_F(x,\delta_2)$ for every $0\leq\delta_1\leq \delta_2$.
		\item $Q_{z+F}(z+x,\delta)=z+Q_F(x,\delta)$ for every $\delta\geq 0$ and $z\in X$.
		\item $Q_{\alpha F}(\alpha x,\delta)=\alpha Q_{F}\left(x,\frac{\delta}{\vert \alpha \vert}\right)$ for every $\delta\geq 0$ and $\alpha \in \mathbb{R}\setminus \{0\}$.
	\end{enumerate}
\end{proposition}

The following result is a consequence of \Cref{Q_F} and easy to verify.  Let $\mathsf{P}$ be any of the property given in \Cref{defn}.

\begin{proposition} Let $A$ be a non-empty subset of $X$ and $F\in CB(X)$. Then the following statements are equivalent.
	\begin{enumerate}
		\item $F$ has property $\mathsf{P}$ on $A$.
		\item $\alpha F$ has property $\mathsf{P}$ on $\alpha A$ for some (hence, for every) $\alpha\in \mathbb{R}\setminus\{0\}$.
		\item $z+F$ has property $\mathsf{P}$ on $z+A$ for some (hence, for every) $z\in X$. 	
	\end{enumerate}
\end{proposition}

Now, we present some conditions which will lead to the union of two sets possesses the property $\mathsf{P}.$ To see this we need the following lemma.

\begin{lemma}\label{farunion} Let $F_1,F_2\in CB(X)$ and $x\in X$.
	If $r(F_1,x)\neq r(F_2,x)$, then for every $0\leq \delta<\vert r(F_1,x)-r(F_2,x)\vert$ we have \begin{center}
		$Q_{F_1\cup F_2}(x,\delta)=\begin{cases}
			Q_{F_1}(x,\delta), & if\ r(F_1,x)> r(F_2,x);\\
			Q_{F_2}(x,\delta), & if\ r(F_1,x)< r(F_2,x).
		\end{cases}$
\end{center}
	Further, if $r(F_1,x)=r(F_2,x)$, then for every $\eta\geq0$, we have
\begin{center}
		$Q_{F_1\cup F_2}(x,\eta)= Q_{F_1}(x,\eta)\cup Q_{F_2}(x,\eta)$.	
\end{center} 	
\end{lemma}

\begin{proof}
First, we observe that $r(F_1\cup F_2,x)=\max\{r(F_1,x),r(F_2,x)\}$.  Since $r(F_i,x)\leq r(F_1\cup F_2,x)$ for $i=1,2$, we have $\max\{r(F_1,x),r(F_2,x)\}\leq r(F_1\cup F_2,x)$. If $y\in F_1\cup F_2$, then $y\in F_1$ or $y\in F_2$ which leads to $r(F_1\cup F_2,x)\leq \max\{r(F_1,x),r(F_2,x)\}$. Hence the equality holds.

\par	Let $r(F_1,x)> r(F_2,x)$. Then $r(F_1\cup F_2, x)=r(F_1,x)$ and clearly for every $\delta\geq0$ we have $Q_{F_1}(x,\delta)\subseteq Q_{F_1\cup F_2}(x,\delta).$ Let $0\leq\delta<r(F_1,x)-r(F_2,x)$ and $y\in Q_{F_1\cup F_2}(x,\delta)$. Then, $\|y-x\|\geq r(F_1\cup F_2,x)-\delta =r(F_1,x)-\delta$. If $y\in F_1$ it is clear that $y\in Q_{F_1}(x,\delta)$. Suppose $y\in F_2$, then  $ r(F_1,x)-\delta < \|y-x\| \leq r(F_2, x) < r(F_1, x)-\delta$, which is a contradiction. Thus,  $Q_{F_1\cup F_2}(x,\delta)\subseteq Q_{F_1}(x,\delta)$. The other case follows in similar lines.

\par If $r(F_1,x)= r(F_2,x),$ then $Q_{F_1}(x,\eta)\cup Q_{F_2}(x,\eta) \subseteq Q_{F_1\cup F_2}(x,\eta)$ for every $\eta\geq0$. Let $y\in Q_{F_1\cup F_2}(x,\eta)$, then  $\|y-x\|\geq r(F_i,x)-\eta$  if $y\in F_i$ for $i=1,2$. Thus, $y\in Q_{F_1}(x,\eta)\cup Q_{F_2}(x,\eta)$. Hence the equality holds.
\end{proof}

\begin{theorem}\label{union}
Let $A$ be a non-empty subset of $X$ and $F_1,F_2\in CB(X)$. Then the following statements hold.
	\begin{enumerate}
		\item If $r(F_1,x)>r(F_2,x)$ for each $x\in A$ and $F_1$ has property $\mathsf{P}$(except USUR) on $A$, Then $F_1\cup F_2$ has property $\mathsf{P}$(except USUR) on $A$. In addition, if   $  \;\inf\{r(F_1,x)-r(F_2,x): x \in A\} >0$ and $F_1$ is USUR on $A$, then $F_1\cup F_2$ is USUR on $A$.

		\item If $r(F_1,x)=r(F_2,x)$ for each $x\in A$ and $F_1, F_2$ are sup-compact (respectively, strongly remotal, remotal) on $A$, then $F_1\cup F_2$ is sup-compact (respectively, strongly remotal, remotal) on $A$.
	\end{enumerate}
\end{theorem}

\begin{proof}
$(1)$: Here we prove for the property $SUR$, the proofs for other properties are similar. Let $\epsilon>0$. Since $F_1$ is $SUR$ on $A$, there exists a $\delta_1>0$ such that $Q_{F_1}(x,\delta_1)\subseteq Q_{F_1}(x)+\epsilon B_X$ for every $x\in A$. Choose $0< \delta<\min\{\delta_1,r(F_1,x)-r(F_2,x)\}$. Now, by \Cref{Q_F} and \Cref{farunion}, we have \begin{center}
$Q_{F_1\cup F_2}(x,\delta)=Q_{F_1}(x,\delta)\subseteq Q_{F_1}(x)+\epsilon B_X= Q_{F_1\cup F_2}(x)+\epsilon B_X.$
\end{center}
Thus, $F_1\cup F_2$ is SUR on $A$.

\noindent $(2)$: Let $F_1, F_2$ are strongly remotal on $A$, $\epsilon>0$ and $x \in A$. Thus, there exists a $\delta>0$ such that $Q_{F_i}(x,\delta)\subseteq Q_{F_i}(x)+\epsilon B_X$ for $i=1,2$. Now, by \Cref{farunion}, we have
\begin{align*}
Q_{F_1\cup F_2}(x,\delta)  & =  Q_{F_1}(x,\delta) \cup Q_{F_2}(x,\delta) \\
							   &  \subseteq  (Q_{F_1}(x)\cup Q_{F_2}(x))+\epsilon B_X \\
							  & = Q_{F_1\cup F_2}(x)+\epsilon B_X.
\end{align*}
Thus, $F_1\cup F_2$ is strongly remotal on $A$. In a similar way, the proof for other properties follows.
\end{proof}
In general, the statement $(2)$ of \Cref{union} does not hold for the properties USUR, SUR and unique remotality. However, as a consequence of \Cref{farunion} and \Cref{union} we have the following observation.
\begin{corollary}
Let $A$ be a non-empty subset of $X$ and $F_1,F_2\in CB(X)$ such that $Q_{F_1}(x)=Q_{F_2}(x)$ for each $x\in A$ and $F_1, F_2$ are $USUR$ (respectively, SUR, uniquely remotal) on $A$. Then $F_1\cup F_2$ is $USUR$(respectively, SUR, uniquely remotal) on $A$. 	
\end{corollary}

The following two results are essential to obtain characterizations for certain rotundity properties in the subsequent section. For any $A\in CL(X)$ and $\delta\geq 0$, the set of nearly best approximants from $A$ to $x$ is defined as $P_A(x,\delta)=\{y\in A: \|x-y\|\leq\inf_{z\in A}\|x-z\|+\delta\}$. We write $P_A(x,0)=P_A(x)$.

\begin{proposition}\label{farclose}
Let $x\in X$ and $\delta \in [0,2]$. Then the following statements hold.
	\begin{enumerate}
		\item $P_{B_X}(x,\delta)\subseteq Q_{B_X}(-x,\delta),$ whenever $\|x\|\geq 1.$		
		\item $P_{S_X}(x,\delta)\subseteq Q_{S_X}(-x,\delta).$
		\item $Q_{S_{X}}(x)=Q_{B_{X}}(x).$
	\end{enumerate}
\end{proposition}

\begin{proof}	
$(1)$: Let $\|x\|\geq1$ and $y\in P_{B_X}(x,\delta).$ Since
	$$\|x+y\|=\|x+y+x-x\| \geq 2\|x\|-\|x-y\| \geq 1+\|x\|-\delta,$$	
	we have $y\in Q_{B_X}(-x,\delta).$

\noindent$(2)$: Let $\|x\|\geq1$. Since $P_{S_X}(z,\delta)=P_{B_X}(z,\delta)\cap S_X$ for any $z \in X$ and $\delta \geq 0$, the proof follows similar to the proof of $(1)$. Let $\|x\|<1$ and $y\in P_{S_X}(x,\delta).$ Since
	$$ \|x+y\|=\|x-y+2y\|	\geq 2\|y\|-\|x-y\|	\geq 1+\|x\|-\delta.$$	Thus, we have $y\in Q_{S_X}(-x,\delta)$.

\noindent$(3)$: It is clear that $Q_{S_X}(x)\subseteq Q_{B_X}(x)$ for any $x \in X.$ Let $y \in Q_{B_X}(x).$ Since, $1+\|x\| = \|x-y\| \leq  \|x\| + \|y\| \leq \|x\|+1,$ we have $y \in S_X.$ Thus, $Q_{S_{X}}(x)=Q_{B_{X}}(x).$
\end{proof}

In general, the containments in the preceding proposition are proper in any rotund space whenever $\delta >0$.

\begin{lemma}\label{cont}
For any $x\in S_X$ and $\delta \geq 0$, the following statements hold.
	\begin{enumerate}
		\item $Q_{B_{X}}(t_2x,\delta)\subseteq Q_{B_{X}}(t_1x,\delta)\subseteq Q_{B_{X}}\left(t_2x,\frac{t_2}{t_1}\delta\right)$, whenever $0<t_1<t_2.$
		\item $Q_{S_{X}}(t_2x,\delta)\subseteq Q_{S_{X}}(t_1x,\delta)\subseteq Q_{S_{X}}\left(t_2x,\frac{t_2}{t_1}\delta\right)$, whenever $0<t_1<t_2.$
	\end{enumerate}
\end{lemma}

\begin{proof}
Since $Q_{S_X}(w, \eta) = Q_{B_X}(w, \eta) \cap S_X$ for any $w \in X$ and $\eta \geq 0,$ it is enough to prove the statement $(1)$. Let $0<t_1<t_2$ and $y\in Q_{B_{X}}(t_2x,\delta)$. Since $$t_2+1-\delta\leq \|y-t_2x\|\leq\|y-t_1x\|+\|t_1x-t_2x\|\leq t_2+1,$$
	we have $t_1+1-\delta\leq\|y-t_1x\|\leq t_1+1,$ i.e. $y\in Q_{B_{X}}(t_1x,\delta).$ To see the other inequality let $z\in Q_{B_{X}}(t_1x,\delta).$ Since
\begin{align*}
t_1+1-\delta \leq \left\|z-t_1x\right\| & \leq \left\|z-\frac{t_1}{t_2}z\right\| + \left\|\frac{t_1}{t_2}z-t_1x\right\| \\
	& \leq \frac{t_2-t_1}{t_2}+\frac{t_1}{t_2}+t_1=1+t_1,
\end{align*}
we have $\frac{t_1}{t_2}+t_1-\delta\leq\|\frac{t_1}{t_2}z-t_1x\|\leq \frac{t_1}{t_2}+t_1.$ Thus, $1+t_2-\frac{t_2}{t_1}\delta\leq\|z-t_2x\|$ i.e. $z\in Q_{B_{X}}\left(t_2x,\frac{t_2}{t_1}\delta\right).$
\end{proof}

The following result is a direct consequence of \Cref{cont}.

\begin{corollary}\label{CAP}
For any $\alpha,\beta>0$ and $x \in S_X$ the following statements hold.
	\begin{enumerate}
		\item $(y_n)$ is a maximizing sequence in $B_X$(respectively, $S_X$) for $\alpha x$ if and only if $(y_n)$ is a maximizing sequence in $B_X$(respectively, $S_X$) for $\beta x$.
		\item  $B_X$ has property $\mathsf{P}$ on $\alpha S_X$ if and only if $B_X$ has property $\mathsf{P}$ on $\beta S_X$.
		\item  $S_X$ has property $\mathsf{P}$ on $\alpha S_X$ if and only if $S_X$ has property $\mathsf{P}$ on $\beta S_X$.
	\end{enumerate}
\end{corollary}	

\begin{proposition}\label{BallSphereEqui}
Let $A$ be a non-empty subset of $X.$ Then, $B_X$ has property $\mathsf{P}$ on $A$ if and only if $S_X$ has property $\mathsf{P}$ on $A.$
\end{proposition}

\begin{proof}
Let $x \in A.$ The property remotality is obvious and unique remotality follows from \Cref{farclose}.

Let $B_X$ be strongly remotal at $x$ and $\epsilon >0.$ Then, there exists $\delta >0$ such that $Q_{B_X}(x, \delta) \subseteq Q_{B_X}(x)+\epsilon B_X.$ Since, $Q_{S_X}(x, \delta) = S_X \cap Q_{B_X}(x, \delta)$ for every $\delta >0$ and $Q_{S_X}(x) = Q_{B_X}(x),$ we have $Q_{S_X}(x, \delta) \subseteq Q_{S_X}(x)+\epsilon B_X.$ i.e $S_X$ is strongly remotal at $x.$ To see the converse, assume $S_X$ is strongly remotal at $x$ and $\epsilon >0.$ Then, there exists $0< \delta <\frac{\epsilon}{2}$ such that $Q_{S_X}(x, \delta) \subseteq Q_{S_X}(x)+\frac{\epsilon}{2} B_X.$ Now, we claim that $Q_{B_X}\left(x, \frac{\delta}{2}\right) \subseteq Q_{B_X}(x)+\epsilon B_X.$ Let $ y \in Q_{B_X}\left(x,\frac{\delta}{2}\right).$ Note that $\|y\| \geq 1-\frac{\delta}{2}.$ Since,
$$ \left\Vert \frac{y}{\|y\|}  -x\right\Vert \geq \left\vert \|x-y\| -  \left\Vert \frac{y}{\|y\|}-y \right\Vert \right\vert  \geq \|x\|+\|y\|-\delta \geq 1+\|x\|-\delta,$$
we have $\frac{y}{\|y\|} \in Q_{S_X}(x, \delta).$ Thus, by assumption, there exists $z \in Q_{S_X}(x)$ such that $\left\Vert \frac{y}{\|y\|}-z\right\Vert \leq \frac{\epsilon}{2}.$ Further, we have
$$\|y-z\| \leq \left\Vert y- \frac{y}{\|y\|}\right\Vert + \left\Vert \frac{y}{\|y\|}-z\right\Vert \leq 1-\|y\|+\frac{\epsilon}{2} < \epsilon.$$ Hence, by \Cref{farclose}, we have $y \in Q_{B_X}(x)+\epsilon B_X.$ i.e.  $B_X$ is strongly remotal at $x.$

Let $B_X$ be sup-compact at $x$ and $(x_n)$  be a maximizing sequence in $S_X$ for $x.$ It is clear that $(x_n)$  be a maximizing sequence in $B_X$ for $x.$ Thus, $(x_n)$ has a convergent subsequence. Therefore, $S_X$ is sup-compact at $x$. To see the converse, let $(y_n)$ be a maximizing sequence in $B_X$ for $x$. Observe that $\|y_n\| \to 1.$ Since for every $n \in \mathbb{N},$
$$\left\vert \left\Vert\frac{y_n}{\|y_n\|}-y_n\right\Vert - \|y_n-x\|\right\vert \leq  \left\Vert\frac{y_n}{\|y_n\|}-x\right\Vert \leq 1+\|x\|,$$
we have $\left(\frac{y_n}{\|y_n\|}\right)$  is a maximizing sequence in $S_X$ for $x.$ Therefore,  $\left(\frac{y_n}{\|y_n\|}\right)$ has a convergent subsequence, which further leads to $(y_n)$ has a convergent subsequence. Hence, $B_X$ is sup-compact at $x$.

The equivalence of the sup-compactness also follows from the equivalence of strong remotality and using \Cref{supcompact} and \Cref{farclose}. The proof of property SUR and USUR are similar to the proofs of sup-compactness and strong remotality respectively.
\end{proof}
Hence, from the preceding two results, we have $B_X$ has property $\mathsf{P}$ on $\alpha S_X$ if and only if $S_X$ has property $\mathsf{P}$ on $\beta S_X$ for any $\alpha, \beta >0.$

We need the following generalized diameter notion to obtain various characterizations for rotundity properties in the next section.  The generalized diameter of a pair of sets $A, B\in CB(X)$ is defined as $r(A,B)=\sup \{\|x-y\|: x\in A,\ y\in B\}.$ The map $r(\cdot,\cdot)$ satisfies the following list of properties which are immediate from its definition.

\begin{proposition}\label{GD-properties} Let $A,B\in CB(X)$. Then we have
	\begin{enumerate}
		\item $r(A,B)=0\Leftrightarrow A=B=\{z\}$ for some $z\in X$.
		\item $r(x+A,x+B)=r(A,B)$ for every $x\in X$. (translation invariance)
		\item $r(kA,kB)=|k|r(A,B)$ for every $k\in \mathbb{R}$. (positive homogeneity)
		\item $r(A,B)=r(B,A)$. (symmetry)
		\item $r(A,B)\leq r(A,C)+r(C,B)$ for any $C\in CB(X)$. (triangular inequality)
		\item $r(A,B)\leq r(C,B)$ whenever $A\subseteq C$ and $C\in CB(X)$. (monotonicity)
		\item $r(\overline{co}(A),\overline{co}(B))=r(A,B)$. (invariance under closed convex hull)
		\item $diam(A)\leq r(A,B)\leq diam(B)$ whenever $A\subseteq B$.
		\item $diam(A\cup B)=\max \{ diam(A), diam(B), r(A,B)\}.$
	\end{enumerate}
\end{proposition}

\section{Characterizations of rotundity properties}\label{sec3}

In this section, we present some characterizations of various rotundity notions. It is clear that the closed unit ball and the unit sphere are remotal on the entire space. In the following results we prove that the unique remotality of the unit sphere characterizes the rotund spaces. For any $x,y\in X$, $[x,y]$ denotes the line segment joining $x$ and $y$.

\begin{theorem}\label{rotund}
	The following statements are equivalent.
	\begin{enumerate}
		\item $X$ is rotund.
		\item $S_X$ is uniquely remotal on $X\setminus\{0\}.$
		\item $S_X$ is uniquely remotal on $X\setminus B_{X}.$
		\item $S_X$ is uniquely remotal on $B_{X}\setminus\{0\}.$
		\item  $S_X$ is uniquely remotal on $B_{X}\setminus k B_X$ for any(some) $0<k<1$.
		\item $S_X$ is uniquely remotal on $S_X$.
	\end{enumerate}
\end{theorem}

\begin{proof}
$(1)\Rightarrow (2)$: Let $x \in X\setminus\{0\}$ and $y\in Q_{S_X}(x)$, i.e. $\|x-y\|=1+\|x\|.$ Since $X$ is rotund, by \cite[Proposition 5.1.11]{Megginson}, there exists $\alpha\geq 0$ such that $-y=\alpha x$, which leads to $y=\frac{-x}{\|x\|}$. Hence the implication holds.

\noindent$(2)\Rightarrow (3)$: Obvious.

\noindent$(3)\Rightarrow (1)$: Suppose X is not rotund, there exist $x_1,x_2\in S_X$ such that $x_1\neq x_2$ and $[x_1,x_2]\subseteq S_X.$ For any $t \in [0, 1]$ we have
$$\|tx_1+(1-t)x_2 + (x_1+x_2)\|=3\left\Vert\left(\frac{1+t}{3}\right)x_1+\left(1-\frac{1+t}{3}\right)x_2 \right\Vert=3.$$
Thus, $[x_1,x_2]\subseteq Q_{S_X}\left(-(x_1+x_2)\right)$, which is a contradiction.

\noindent $(2)\Rightarrow(4)\Rightarrow(5)\Rightarrow(6)$: Obvious.

\noindent $(6)\Rightarrow(1)$: Suppose $X$ is not rotund, there exist $x_1,x_2\in S_X$ such that $x_1\neq x_2$ and $[x_1,x_2]\subseteq S_X$. For any $t \in [0, 1]$ we have
$$\left\Vert tx_1+(1-t)x_2 + \frac{x_1+x_2}{2}\right\Vert = 2\left\Vert\left(\frac{1+2t}{4}\right)x_1+\left(1-\frac{1+2t}{4}\right)x_2 \right\Vert=2.$$
Thus, $[x_1,x_2]\subseteq Q_{S_X}\left(-\frac{x_1+x_2}{2}\right)$, which is a contradiction.
\end{proof}



The following characterization of rotund spaces in terms of Chebyshevness is a direct consequence of \Cref{farclose,BallSphereEqui} and \Cref{rotund}.
\begin{corollary} Let $0 < \alpha (\neq 1)$ and $\beta >1$. Then,
the  space $X$ is rotund if and only if $S_X$ is Chebyshev on $\alpha S_X$ if and only if $B_X$ is Chebyshev on $\beta S_X.$
\end{corollary}

 The equivalence of statements $(2)$ to $(6)$ in \Cref{rotund} also follows from \Cref{CAP}. Now, we present another characterization of rotund spaces in terms of generalized diameter.

\begin{theorem}\label{rotund2}
	The following statements are equivalent.
	\begin{enumerate}
		\item $X$ is rotund.
		\item $r\left(Q_{B_X}(x_1), Q_{B_X}(x_2)\right)=\|x_1-x_2\|$ for every $x_1, x_2\in S_X$.
		\item $r\left(Q_{S_X}(x_1), Q_{S_X}(x_2)\right)=\|x_1-x_2\|$ for every $x_1, x_2\in S_X$.
       	\item $r\left(Q_{S_X}(x_1), Q_{S_X}(x_2)\right)=\|x_1-x_2\|$ for every $x_1, x_2\in S_X$ with $x_1\neq x_2$.
	\end{enumerate}
\end{theorem}

\begin{proof}
$(1)\Rightarrow (2) \Rightarrow (3) \Rightarrow (4)$: These implications follow from \Cref{farclose} and \Cref{rotund}.



\noindent$(4)\Rightarrow (1)$: Suppose $X$ is not rotund. Then there exist $x_1,x_2\in S_X$ such that $x_1\neq x_2$ and $[x_1, x_2] \subseteq S_X$. Let $y_1=-\frac{x_1+3x_2}{4}$ and $y_2=-\frac{3x_1+x_2}{4}.$ It is clear that $y_1 \in S_X$ and $ \Vert x_1 - y_1 \Vert = \left\Vert \frac{5x_1+3x_2}{4}\right\Vert=2.$ Therefore, we have $x_1 \in Q_{S_X}(y_1).$ Similarly, $x_2 \in Q_{S_X}(y_2).$ Thus, $\Vert y_1-y_2\Vert < r(Q_{S_X}(y_1), Q_{S_X}(y_2)),$ which is a contradiction.
\end{proof}

It is easy to observe that the space $X$ is of finite dimension if and only if its unit sphere $S_X$ is sup-compact on $\{0\}.$ However, in the following result we see that if the unit sphere is sup-compact at every non-zero element then the space is compactly locally uniformly rotund and vice versa.

\begin{theorem}\label{clur-char}
	The following statements are equivalent.
	\begin{enumerate}
		\item $X$ is $CLUR$.
		\item $S_X$ is sup-compact on $X\setminus \{0\}$.
		\item $Q_{S_X}(x)$ is compact and $Q_{S_X}\left(x,\frac{1}{n}\right) \xrightarrow{V} Q_{S_X}(x)$ for every $x\in X\setminus \{0\}$.
		\item $Q_{S_X}(x)$ is compact and $Q_{S_X}\left(x,\frac{1}{n}\right) \xrightarrow{H} Q_{S_X}(x)$ for every $x\in X\setminus \{0\}$.
		\item $Q_{S_X}(x)$ is compact and $S_X$ is strongly remotal on $X\setminus \{0\}$.
		\item $\alpha\left(Q_{S_X}\left(x,\frac{1}{n}\right)\right)\to 0$ for every $x\in X\setminus \{0\}$.
	\end{enumerate}
\end{theorem}

\begin{proof}
$(1)\Rightarrow (2)$: Let $(x_n)$ be a maximizing sequence in $S_X$ for $x\in X\setminus\{0\}$. Then, it follows from  \Cref{CAP}  that $(x_n)$ is a maximizing sequence in $S_X$ for $\frac{x}{\|x\|}$. Since $X$ is $CLUR$ and $\left\|x_n+\frac{(-x)}{\|x\|}\right\|\to 2,$ there exists a  subsequence $(x_{n_k})$ of $(x_n)$ such that $x_{n_k} \to y$ for some $y\in S_X$. Therefore, $S_X$ is sup-compact at every $x \in X \setminus \{0\}.$


\noindent$(2)\Rightarrow (1)$: Let $(x_n)$ be a sequence in $S_X$ and $x\in S_X$ such that $\left\|\frac{x_n+x}{2}\right\|\to 1$. Notice that $(x_n)$ is a maximizing sequence in $S_X$ for $-x\in S_X$. Thus, by assumption, $(x_n)$ has a convergent subsequence. i.e $X$ is CLUR.

\noindent$(2)\Leftrightarrow (3)\Leftrightarrow(4)\Leftrightarrow(5)\Leftrightarrow(6)$:  These implications follow from \Cref{supcompact}.
\end{proof}


The following characterization \cite[Theorem 3.14]{thota} can be obtained as a consequence of \Cref{farclose} and \Cref{clur-char}.
\begin{corollary}
The space $X$ is CLUR if and only if $S_X$ is approximatively compact on $X \setminus \{0\}.$
\end{corollary}

Now, we present a characterization for local uniform rotundity in terms of strong unique remotality of the unit sphere. For this we need the following lemma.

\begin{lemma}\label{GD-ineq}
Let $x, x' \in X \setminus \{0\}$ and $l= \max\{1, \frac{1}{\|x\|}, \frac{1}{\|x'\|}\}.$ Then for any $\delta >0$ we have
$$ r(Q_{S_X}\left(x, \delta), Q_{S_X}(x', \delta)\right) \leq r\left(Q_{S_X}\left(\frac{x}{\|x\|}, l\delta\right),  Q_{S_X}\left(\frac{x'}{\|x'\|}, l\delta\right)\right). $$
\end{lemma}

\begin{proof}
Let $x\in X\setminus \{0\}$ and $\delta >0$. If $\|x\|\leq 1$, then, by \Cref{cont}, we have
$$Q_{S_X}\left(x, \delta\right) \subseteq Q_{S_X}\left(\frac{x}{\|x\|}, \frac{\delta}{\|x\|}\right).$$
If $1< \|x\|$, then, by \Cref{cont}, it follows that
$$Q_{S_X}\left(x,\delta\right)= Q_{S_X}\left(\|x\|\frac{x}{\|x\|}, \delta\right)\subseteq Q_{S_X}\left(\frac{x}{\|x\|}, \delta\right).$$
Thus, for any $x,x'\in X\setminus \{0\}$ and $l=\max\{1,\frac{1}{\|x\|},\frac{1}{\|x'\|}\}$, we have
$$Q_{S_X}\left(x, \delta\right)\subseteq  Q_{S_X}\left(\frac{x}{\|x\|}, l\delta\right) \mbox{and } Q_{S_X}\left(x', \delta\right)\subseteq  Q_{S_X}\left(\frac{x'}{\|x'\|}, l\delta\right).$$
Hence, the inequality follows from \Cref{GD-properties}.
\end{proof}

\begin{theorem}\label{LURT}
	The following statements are equivalent.
	\begin{enumerate}
		\item $X$ is $LUR.$
		\item $r\left(Q_{S_X}\left(x,\frac{1}{n}\right), Q_{S_X}\left(x',\frac{1}{n}\right)\right)\to \left\|\frac{x}{\|x\|}-\frac{x'}{\|x'\|}\right\|$ for every $x,x'\in X\setminus\{0\}$.		
		\item $diam \left(Q_{S_{X}}\left(x,\frac{1}{n}\right)\right)\to 0$ for every $x\in X\setminus\{0\}.$
		\item Every maximizing sequence  in $S_X$ for $x\in X\setminus\{0\}$ converges to $\frac{-x}{\|x\|}$.
		\item $Q_{S_X}(x)=\left\{\frac{-x}{\|x\|}\right\}$ and $Q_{S_X}\left(x,\frac{1}{n}\right) \xrightarrow{V} Q_{S_X}(x)$ for every $x\in X\setminus\{0\}$.
		\item $Q_{S_X}(x)=\left\{\frac{-x}{\|x\|}\right\}$ and $Q_{S_X}\left(x,\frac{1}{n}\right)  \xrightarrow{H}Q_{S_X}(x)$ for every $x\in X\setminus\{0\}$.
		\item $S_X$ is SUR on $X\setminus\{0\}$.
	\end{enumerate}
\end{theorem}

\begin{proof}
$(1)\Rightarrow(2)$: Notice that $\left\|\frac{x}{\|x\|}-\frac{x'}{\|x'\|}\right\| \leq r\left(Q_{S_X}\left(x,\frac{1}{n}\right), Q_{S_X}\left(x',\frac{1}{n}\right)\right),$ for every $x, x' \in X$ and $n \in \mathbb{N}.$ Thus, in view of \Cref{GD-ineq}, it is enough to show that $r \left(Q_{S_{X}}\left(y,\frac{1}{n}\right),Q_{S_X}\left(y',\frac{1}{n}\right)\right)$ converges to $\|y-y'\|$ on $S_X\times S_X$. Let $y,y'\in S_X$ and $(\delta_n)$ be a decreasing sequence of real numbers converging to $0$. Now, we claim that  $r \left(Q_{S_{X}}\left(y,\delta_n\right), Q_{S_X}\left(y',\delta_n\right)\right)$ converges to $\|y-y'\|$. Suppose there exists $\epsilon >0$ such that
\begin{center}
$r\left(Q_{S_{X}}(y,\delta_n),Q_{S_{X}}(y',\delta_n)\right)-\|y-y'\| \geq \epsilon$ for every $n\in \mathbb{N}$.	
\end{center}
Then there exist two sequences $(y_n),(y_n')$ in $S_X$ such that  for every $n \in \mathbb{N},$  $y_n\in Q_{S_{X}}(y,\delta_n),  y_n'\in Q_{S_{X}}(y',\delta_n)$ and $\|y_n-y_n'\|-\|y-y'\|>\frac{\epsilon}{2}$. Since $X$ is $LUR$, we have $y_n\to -y$ and $y_n'\to -y'$.
\begin{align*}
\frac{\epsilon}{2}\leq \|y_n-y_n'\|-\|y-y'\|&\leq \|y_n+y-y+y'-y'-y_n'\|-\|y-y'\|\\
		&\leq\|y_n+y\|+\|y'+y_n'\|,
\end{align*}
which is a contradiction. Thus, $r\left(Q_{S_{X}}(y,\delta_n),Q_{S_{X}}(y',\delta_n)\right)\to \|y-y'\|$ for every $y,y'\in S_X$. Hence the implication holds.

\noindent$(2)\Rightarrow (3)$: This implication follows by considering $x=x'$.

\noindent$(3)\Leftrightarrow(4)\Leftrightarrow(5)\Leftrightarrow(6)\Leftrightarrow(7)$: These implications follow from \Cref{remseq}.

\noindent$(4)\Rightarrow(1)$: Let $(x_n)$ be a sequence in $S_X$ and $x\in S_X$ such that $\|\frac{x_n+x}{2}\|\to 1$. Observe that $(x_n)$ is a maximizing sequence for $-x$. Thus, by assumption, we have $x_n\to x$.
\end{proof}

We remark that the notion strong unique remotality is different from strongly Chebyshev \cite{band}. For example, $B_X$ is strongly Chebyshev on $X\setminus\{0\}$ if and only if the space is midpoint locally uniformly rotund \cite[Theorem 5.3.28]{Megginson}. But, as a consequence of \Cref{BallSphereEqui} and \Cref{LURT}, we have $B_X$ is strongly uniquely remotal on $X\setminus\{0\}$ if and only if $X$ is LUR.	

In the following result we present a characterization for $LUR$ spaces, in terms of the sets $P_{S_X}(\cdot,\cdot)$ using generalized diameter. The proof follows in similar lines to the proof of \Cref{LURT}, also using \Cref{farclose} and \cite[Theorem 3.14]{thota}.

\begin{corollary}
	The following statements are equivalent.
	\begin{enumerate}
		\item $X$ is $LUR.$	
		\item $r\left(P_{S_X}\left(x,\frac{1}{n}\right), P_{S_X}\left(x',\frac{1}{n}\right)\right)\to \left\|\frac{x}{\|x\|}-\frac{x'}{\|x'\|}\right\|$ for every $x,x'\in X\setminus\{0\}$.		
		\item $diam \left(P_{S_{X}}\left(x,\frac{1}{n}\right)\right)\to 0$ for every $x\in X\setminus\{0\}.$
		\item Every minimizing sequence  in $S_X$ for $x\in X\setminus\{0\}$ converges to $\frac{x}{\|x\|}$.
		\item $P_{S_X}(x)=\left\{\frac{x}{\|x\|}\right\}$ and $P_{S_X}\left(x,\frac{1}{n}\right)  \xrightarrow{V} P_{S_X}(x)$ for every $x\in X\setminus\{0\}$.
		\item $P_{S_X}(x)=\left\{\frac{x}{\|x\|}\right\}$ and $P_{S_X}\left(x,\frac{1}{n}\right)  \xrightarrow{H} P_{S_X}(x)$ for every $x\in X\setminus\{0\}$.
		\item $S_X$ is strongly Chebyshev on $X\setminus\{0\}$.
	\end{enumerate}
\end{corollary}

Now, we proceed to characterize uniform rotundity in terms of the sets of almost farthest points $Q_{S_X}(\cdot,\cdot).$ For any $\alpha >0$, we denote $X_\alpha = X\setminus \alpha B_X.$

\begin{theorem}\label{URT}
	Let $\alpha>0$. Then the following statements are equivalent.
	\begin{enumerate}
		\item $X$ is $UR$.
		\item $r\left(Q_{S_X}\left(x,\frac{1}{n}\right),Q_{S_X}\left(x',\frac{1}{n}\right)\right)\to \left\|\frac{x}{\|x\|}-\frac{x'}{\|x'\|}\right\|$ uniformly on $X_\alpha \times X_\alpha$.
		\item $diam \left(Q_{S_{X}}\left(x,\frac{1}{n}\right)\right)\to 0$ uniformly on $x\in X_\alpha$.
		\item $Q_{S_X}(x)=\left\{\frac{-x}{\|x\|}\right\}$ and $Q_{S_X}\left(x,\frac{1}{n}\right) \xrightarrow{H} Q_{S_X}(x)$ uniformly on $x\in X_\alpha.$
		\item $S_X$ is USUR on $X_\alpha$.
	\end{enumerate}
\end{theorem}

\begin{proof}
$(1)\Rightarrow(2)$: 
Let $y \in X_\alpha$ and $l=\max\{1,\frac{1}{\alpha}\}.$  Then, it follows from the proof of \Cref{GD-ineq} that for every $n \in \mathbb{N}$ $$Q_{S_X}\left(y,\frac{1}{n}\right)\subseteq  Q_{S_X}\left(\frac{y}{\|y\|},\frac{l}{n}\right)$$
holds. Thus, for any $x, x'\in X_\alpha$ we have
\begin{align*}
		\left\|\frac{x}{\|x\|}-\frac{x'}{\|x'\|}\right\|&\leq r\left(Q_{S_X}\left(x,\frac{1}{n}\right),Q_{S_X}\left(x',\frac{1}{n}\right)\right)\\
		&\leq r\left(Q_{S_X}\left(\frac{x}{\|x\|},\frac{l}{n}\right),Q_{S_X}\left(\frac{x'}{\|x'\|},\frac{l}{n}\right)\right).
\end{align*}
Therefore, it is enough to show that $r\left(Q_{S_{X}}\left(y,\frac{1}{n}\right),Q_{S_X}\left(y',\frac{1}{n}\right)\right)$ converges uniformly to $\|y-y'\|$ on $S_X\times S_X$. Suppose there exist $\epsilon >0$ and two sequences $(y_k),(y_k')$ in $S_X$ such that $$ r\left(Q_{S_{X}}\left(y_k,\frac{1}{k}\right),Q_{S_X}\left(y_k',\frac{1}{k}\right)\right)-\|y_k-y_k'\|\geq \epsilon.$$ for all $k\in \mathbb{N}$. Now, for every $k\in \mathbb{N}$, choose $z_k\in Q_{S_{X}}\left(y_k,\frac{1}{k}\right)$ and $z_k'\in Q_{S_X}\left(y_k',\frac{1}{k}\right)$ such that $\|z_k-z_k'\|-\|y_k-y_k'\|>\frac{\epsilon}{2}$. Since $X$ is $UR$, it is easy to verify that $\|y_k+z_k\|\to 0$ and $\|y_k'+z_k'\|\to 0$. However, for any $k\in \mathbb{N}$, we have
	$$\frac{\epsilon}{2}<\|z_k-z_k'\|-\|y_k-y_k'\|\leq\|z_k+y_k\|+\|z_k'+y_k'\|.$$ which is a contradiction. Thus, the implication holds.

\noindent$(2)\Rightarrow (3)$: This implication follows by considering $x=x'$.

\noindent$(3)\Rightarrow (4)$:
For any $x \in X_\alpha$, by \Cref{remseq}, we have $Q_{S_X}(x)= \left\{\frac{-x}{\|x\|}\right\}.$
To see the other part, let $\epsilon >0$. By assumption, there exists $n_0\in \mathbb{N}$ such that  $diam\left(Q_{S_X}\left(x,\frac{1}{n}\right)\right)<\epsilon$ for every $n\geq n_0$ and $x\in X_\alpha$. Notice that $\left\|y+\frac{x}{\|x\|}\right\|<\epsilon$ whenever $y\in Q_{S_X}\left(x,\frac{1}{n}\right)$, $x\in X_\alpha$ and $n\geq n_0$. Thus, $H\left(Q_{S_X}\left(x,\frac{1}{n}\right), Q_{S_X}(x)\right)<\epsilon$ for every $n\geq n_0$ and $x\in X_\alpha$, the implication holds.

\noindent$(4)\Rightarrow (5)$: Obvious.

\noindent$(5)\Rightarrow (1)$: Suppose there exists $\epsilon>0$, two sequences $(x_n),(y_n)$ in $S_X$ such that $\left\|\frac{x_n+y_n}{2}\right\| \to 1$ but $\|x_n-y_n\| > \epsilon$. Since, by assumption and \Cref{CAP}, there exists $\delta>0$ such that $Q_{S_X}(x,\delta)\subseteq Q_{S_X}(x)+\epsilon B_X$ and $Q_{S_X}(x)=\{-x\}$ for every $x\in S_X$. For this $\delta>0$ there exists $n_0\in \mathbb{N}$ such that $2\geq \|x_{n}-(-y_{n})\|\geq 2-\delta$ for every $n\geq n_0$. Observe that $-y_n\in Q_{S_X}(x_n,\delta)$ for every $n\geq n_0$. Since $\|y_{n}-x_{n}\|>\epsilon$, we have $Q_{S_X}(x_n,\delta)\nsubseteq Q_{S_X}(-x_n)+\epsilon B_X$ for every $n\geq n_0$. This is a contradiction.
\end{proof}

In the following result we present a characterization of $UR$, in terms of sets $P_{S_X}(\cdot,\cdot)$ using generalized diameter. The proof follows in similar lines of the proof of previous theorem, \Cref{farclose} and \cite[Remark 3.4]{UC}. For the definition of uniformly strongly Chebyshev we refer to \cite{UC}.

\begin{corollary}
	The following statements are equivalent.
	\begin{enumerate}
		\item $X$ is $UR$.
		\item $r\left(P_{S_X}\left(x,\frac{1}{n}\right),P_{S_X}\left(x',\frac{1}{n}\right)\right)\to \left\|\frac{x}{\|x\|}-\frac{x'}{\|x'\|}\right\|$ uniformly on $X_1\times X_1$.
		\item $diam \left(P_{S_{X}}\left(x,\frac{1}{n}\right)\right)\to 0$ uniformly on $x\in X_1.$
		\item $P_{S_X}(x)=\left\{\frac{x}{\|x\|}\right\}$ and $P_{S_X}\left(x,\frac{1}{n}\right)  \xrightarrow{H } P_{S_X}(x)$ uniformly on $x\in X_1.$
		\item $S_X$ is uniformly strongly Chebyshev on $X_1$.
	\end{enumerate}
\end{corollary}

\begin{remark} The \Cref{rotund,clur-char,LURT,URT} are also hold with the following modifications.
\begin{enumerate}
	\item  In view of \Cref{CAP}, it is enough to consider the $\alpha S_X$ for some $\alpha >0$ instead of $X \setminus \{0\}$ and  $X_1.$

	\item In view of \Cref{BallSphereEqui}, characterizations of rotund, CLUR, LUR and  UR spaces hold in terms of the closed unit ball as well.
\end{enumerate}
\end{remark}

In the following two results we present necessary conditions for a space to be $UR$ and $LUR$ respectively. 
 The distance between two non-empty closed sets $A$ and $B$ is defined as $d(A,B)=\inf \{\|x-y\|: x\in A,\ y\in B\}.$

\begin{theorem}\label{iUR}
	Consider the following statements.
	\begin{enumerate}
		\item $X$ is $UR$.
		\item $d\left(Q_{S_X}\left(x,\frac{1}{n}\right),Q_{S_X}\left(x',\frac{1}{n}\right)\right)\to \|x-x'\|$ uniformly on $S_X\times S_X$.
		\item $d\left(Q_{S_X}\left(x,\frac{1}{n}\right),Q_{S_X}\left(-x,\frac{1}{n}\right)\right)\to 2$ uniformly on  $x\in S_X$.\end{enumerate}
	Then, $(1)\Rightarrow (2)\Rightarrow (3)$.
\end{theorem}

\begin{proof}
$(1)\Rightarrow (2)$: Suppose $d\left(Q_{S_X}\left(x,\frac{1}{n}\right),Q_{S_X}\left(x',\frac{1}{n}\right)\right)$ does not converge uniformly to $\|x-x'\|$ on $S_X\times S_X$. Then there exist $\epsilon>0$ and two sequences $(x_n),(x_n')$ in $S_X$ such that $$\|x_n-x_n'\|-d\left(Q_{S_X}\left(x_n,\frac{1}{n}\right),Q_{S_X}\left(x_n',\frac{1}{n}\right)\right)>\epsilon,$$ for all $n\in \mathbb{N}$.
	Thus, for every $n\in \mathbb{N}$, there exist $y_n\in Q_{S_X}\left(x,\frac{1}{n}\right)$ and $y_n'\in Q_{S_X}\left(x',\frac{1}{n}\right)$ satisfying
	$$\frac{\epsilon}{2}<\|x_n-x_n'\|-\|y_n-y_n'\| <\|x_n+y_n\|+\|x_n'+y_n'\|.$$
	Since $X$ is $UR$, it follows that $\|x_n+y_n\|\to 0$ and $\|x_n'+y_n'\|\to 0$, which is a contradiction. Hence the implication holds.

\noindent$(2) \Rightarrow (3)$: This implication follows by taking $x'=-x$.
\end{proof}


\begin{theorem}\label{ilur}
	Consider the following statements.
	\begin{enumerate}
		\item $X$ is $LUR$.
		\item $d\left(Q_{S_X}\left(x,\frac{1}{n}\right),Q_{S_X}\left(x',\frac{1}{n}\right)\right)\to \|x-x'\|$ for every $x,x'\in S_X$.
		\item $d\left(Q_{S_X}\left(x,\frac{1}{n}\right),Q_{S_X}\left(-x,\frac{1}{n}\right)\right)\to 2$ for every $x\in S_X$.\end{enumerate}
	Then, $(1)\Rightarrow (2)\Rightarrow (3)$.
\end{theorem}

The following example illustrates that the reverse implications in \Cref{iUR,ilur} need not hold in general.

\begin{example}
Let $X=(\mathbb{R}^2,\|.\|_\infty)$ and $x=(x_1,x_2)\in S_X$. Then, it is easy to see that $d\left(Q_{S_X}\left(x,\frac{1}{n}\right),Q_{S_X}\left(-x,\frac{1}{n}\right)\right)= 2\left(1-\frac{1}{n}\right)$ for all $n \in \mathbb{N}$ and converges to $2$ uniformly on $S_X$. However, $X$ is not $LUR$($UR$).
\end{example}

\section*{Acknowledgments}
The first named author would like to thank the National Institute of Technology Tiruchirappalli, India for the financial support.

\end{document}